\documentclass{amsart}

\usepackage{amsfonts, amssymb, amsmath, eucal, verbatim, amsthm, amscd, enumerate}
\usepackage{setspace}

\usepackage{ulem}

\newtheorem{theorem}{Theorem}[section]
\newtheorem{lemma}[theorem]{Lemma}
\newtheorem{proposition}[theorem]{Proposition}

\newtheorem{corollary}[theorem]{Corollary}

\theoremstyle{definition}

\theoremstyle{remark}
\newtheorem*{remark}{Remark}

\newtheorem*{notation}{Notation}

\newtheorem*{overview}{Overview}

\numberwithin{equation}{section}
\def\R{{\mathbb R}}

\def\N{{\mathbb N}}

\def\p#1{{\left({#1}\right)}}
\def\b#1{{\left\{{#1}\right\}}}

\begin{document}

\title[Optimal forward and reverse estimates with angular smoothing]
 {Optimal forward and reverse estimates of Morawetz and Kato--Yajima
type with angular smoothing index}

\author[N.~Bez]{Neal Bez}
\address{Neal Bez, Department of Mathematics, Graduate School of Science and 
Engineering,
Saitama University, Saitama 338-8570, Japan}
\email{nealbez@mail.saitama-u.ac.jp}

\author[M.~Sugimoto]{Mitsuru Sugimoto}
\address{Mitsuru Sugimoto, Graduate School of Mathematics, Nagoya University \\
Furocho, Chikusa-ku, Nagoya 464-8602, Japan}
\email{sugimoto@math.nagoya-u.ac.jp}




\begin{abstract}
For the solution of the free Schr\"odinger equation, we obtain the optimal 
constants and characterise extremisers for forward and reverse smoothing 
estimates which are global in space and time, contain a homogeneous and
radial weight in the space variable, and incorporate a certain angular regularity. 
This will follow from a more general result which permits analogous
sharp forward and reverse smoothing estimates and a characterisation
of extremisers for the solution of the free Klein--Gordon and wave
equations. The nature of extremisers is shown to be sensitive to both
the dimension and the size of the smoothing index relative to the dimension. 
Furthermore, in four spatial dimensions and certain special values of the
smoothing index, we obtain an exact identity for each of these evolution 
equations.
\end{abstract}

\maketitle \thispagestyle{empty}

\section{Introduction}

For $d \geq 2$ and $s \in (-\frac{1}{2},\frac{d}{2}-1)$ the solution
of the free Schr\"odinger equation $i\partial_t u + \frac{1}{2}\Delta u = 0$
satisfies the smoothing estimate
\begin{equation} \label{e:KY}
\int_{\mathbb{R}} \int_{\mathbb{R}^d} |u(x,t)|^2 \, \frac{\mathrm{d}x\mathrm{d}t}{|x|^{2(1+s)}}  \leq C \|u(0)\|_{\dot{H}^s(\mathbb{R}^d)}^2,
\end{equation}
where $\dot{H}^s(\mathbb{R}^d)$ is the usual homogeneous Sobolev
space of order $s$. This estimate was established by Kato and Yajima
\cite{KatoYajima} for $s \in (-\frac{1}{2},0]$ whenever $d \geq 3$,
and $s \in (-\frac{1}{2},0)$ for $d=2$ (see also \cite{BK} for an
alternative approach, and \cite{Su1}, \cite{Vilela} and
\cite{Watanabe} for the full range $s \in
(-\frac{1}{2},\frac{d}{2}-1)$). Estimates like \eqref{e:KY} are
often referred to as Kato--Yajima smoothing estimates, or Morawetz
estimates, since similar estimates for the Klein--Gordon equation
were established in the earlier work \cite{Morawetz}.

The focus of this paper are certain angular refinements of
\eqref{e:KY}. Hoshiro \cite{Hoshiro} proved that whenever $d \geq 3$
and $s \in (-\frac{1}{2},0]$, there is a finite constant $C$ such
that
\begin{equation} \label{e:Hoshiro}
\int_{\mathbb{R}} \int_{\mathbb{R}^d} |(1-\Lambda)^{\frac{1+2s}{4}}u(x,t)|^2 \, \frac{\mathrm{d}x\mathrm{d}t}{|x|^{2(1+s)}}  \leq C \|u(0)\|_{\dot{H}^s(\mathbb{R}^d)}^2,
\end{equation}
where $-\Lambda$ is the Laplace--Beltrami operator on the unit
sphere $\mathbb{S}^{d-1}$ homogeneously extended to $\mathbb{R}^d$.
In fact, \eqref{e:Hoshiro} is also valid in the range $s \in
(-\frac{1}{2},\frac{d}{2}-1)$ for any $d \geq 2$ (see \cite{Su1} for
the full range). Interestingly, it was recently observed by Fang and
Wang \cite{FangWang} that a reverse form of \eqref{e:Hoshiro}
exists; that is, for the same $(d,s)$ there exists a strictly
positive constant $c$ such that
\begin{equation} \label{e:Hoshiroreverse}
\int_{\mathbb{R}} \int_{\mathbb{R}^d} |(1-\Lambda)^{\frac{1+2s}{4}}u(x,t)|^2 \, \frac{\mathrm{d}x\mathrm{d}t}{|x|^{2(1+s)}}  \geq c \|u(0)\|_{\dot{H}^s(\mathbb{R}^d)}^2.
\end{equation}

In the critical case $s = -\frac{1}{2}$ the estimate \eqref{e:KY}
fails, and the full gain of a half-derivative does not materialise
in this way. One may interpret \eqref{e:Hoshiro} as a replacement
for this false estimate since, formally, $(1-\Lambda)^{\sigma}$
behaves like $|x|^{2\sigma}|\nabla|^{2\sigma}$ in the sense of the
order of the derivative and the decay. A different replacement for
the failure of \eqref{e:KY} when $s = -\frac{1}{2}$ is the local
smoothing estimate
\begin{equation} \label{e:local}
\sup_{R > 0} \frac{1}{R}  \int_\mathbb{R} \int_{|x| \leq R} |\nabla u(x,t)|^2 \, \mathrm{d}x\mathrm{d}t \leq C\|u(0)\|_{\dot{H}^{\frac{1}{2}}(\mathbb{R}^d)}^2
\end{equation}
established in \cite{CS}, \cite{Sjolin} and \cite{Vega}. We remark
that it was recently observed by Vega and Visciglia \cite{VV} that
\eqref{e:local} also enjoys a reverse form; in fact, they prove
\begin{equation*}
\sup_{R > 0} \frac{1}{R}  \int_\mathbb{R} \int_{|x| \leq R} |\nabla u(x,t)|^2 \, \mathrm{d}x\mathrm{d}t \geq 2\pi\|u(0)\|_{\dot{H}^{\frac{1}{2}}(\mathbb{R}^d)}^2.
\end{equation*}
The critical case was also considered in \cite{Su2} and \cite{RS1},
in the context of more general elliptic operators, and applied to
time global existence of solutions to certain derivative nonlinear
equations in \cite{RS3}.

One of our main results in this paper is to compute the optimal
constants and characterise extremisers for the forward and reverse
estimates in \eqref{e:Hoshiro} and \eqref{e:Hoshiroreverse}. These
optimal estimates will follow from a more general result, which we
state first. Our arguments are not only restricted to the
Schr\"odinger propagator, and we consider forward and reverse
estimates of the form
\begin{equation} \label{e:Main}
c  \|u(0)\|_{L^2(\mathbb{R}^d)}^2 \leq
\int_{\mathbb{R}} \int_{\mathbb{R}^d} |\psi(|\nabla|) \,\theta(-\Lambda) \,u(t,x)|^2 \, \frac{\mathrm{d}x\mathrm{d}t}{|x|^\tau}
\leq C  \|u(0)\|_{L^2(\mathbb{R}^d)}^2
\end{equation}
for solutions of $i\partial_t u + \phi(|\nabla|)u = 0$, where $\tau
\in (1,d)$, and $\phi$ and $\psi$ are such that
\begin{equation} \label{e:phipsi}
\psi(\rho)^2 = |\phi'(\rho)|\rho^{1-\tau}.
\end{equation}
Here, we are assuming that the
dispersion relation $\phi$ is injective and differentiable.
\begin{theorem}\label{t:Main}
Let $d \geq 2$ and $\tau \in (1,d)$. Suppose
\[
\beta_k = \pi 2^{2-\tau} \frac{\Gamma(\tau - 1)\Gamma(k+\tfrac{d-\tau}{2})}
{\Gamma(\tfrac{\tau}{2})^2\Gamma(k+\tfrac{d+\tau}{2}-1)} \,|\theta(k(k+d-2))|^2,
\]
and
\[
\mathbf{k} = \{ k \in \mathbb{N}_0 : \inf_{\ell \in \mathbb{N}_0} \beta_\ell = \beta_{k}\} \qquad \text{and}
\qquad \mathbf{K} = \{ k \in \mathbb{N}_0 : \sup_{\ell \in \mathbb{N}_0} \beta_\ell = \beta_{k}\}.
\]
If $i\partial_t u + \phi(|\nabla|)u = 0$ and $\psi(\rho)^2 =
|\phi'(\rho)|\rho^{1-\tau}$ then
\begin{equation*}
\inf_{k\in\N_0} \beta_k \, \|u(0)\|_{L^2(\mathbb{R}^d)}^2 \leq
\int_{\mathbb{R}} \int_{\mathbb{R}^d} |\psi(|\nabla|) \,\theta(-\Lambda) \,u(t,x)|^2 \, \frac{\mathrm{d}x\mathrm{d}t}{|x|^\tau}
\leq \sup_{k\in\N_0} \beta_k \, \|u(0)\|_{L^2(\mathbb{R}^d)}^2
\end{equation*}
and the constants are optimal. Furthermore, nonzero initial data
$u(0)$ is an extremiser for the lower bound if and only if $u(0)$ belongs to $\bigoplus_{k \in \mathbf{k}}
\mathfrak{H}_k$, and an extremiser for the upper bound if and only
if $u(0)$ belongs to $\bigoplus_{k \in
\mathbf{K}} \mathfrak{H}_k$.
\end{theorem}
Here, we are using the notation $\mathbb{N}_0$ for
$\{0,1,2,\ldots\}$, and $\mathfrak{H}_k$ for the space of all linear
combinations of functions
\begin{equation*}
\xi \mapsto P(\xi) f_0(|\xi|) |\xi|^{-d/2-k+1/2}
\end{equation*}
where $P$ is a homogeneous harmonic polynomial of order $k$ and $f_0
\in L^2(0,\infty)$. Also, $\theta(-\Lambda)$ is the homogeneous
extension of the operator $\theta(-\Lambda)$ on the sphere
$\mathbb{S}^{d-1}$; an explicit definition will be given later in
Section \ref{section:TheoremMainProof}.

We remark that $\mathfrak{H}_{0}$ is the
space of square-integrable radially symmetric functions. If the
index set $\mathbf{k}$ is empty then there are no extremisers for
the lower bound, and similarly for $\mathbf{K}$ and the upper bound.

The statement of Theorem \ref{t:Main} is rather general and as a
consequence of the minimal assumptions on $\theta$, the theorem does
not guarantee the strict positivity of $\inf_{k \in \mathbb{N}_0}
\beta_k$ or the finiteness of $\sup_{k \in \mathbb{N}_0} \beta_k$.
The case of primary interest in this paper is
\[
\theta(\rho) = (1+\rho)^{\frac{\tau - 1}{4}}.
\]
For such $\theta$, it is true that $\inf_{k \in \mathbb{N}_0}
\beta_k$ is strictly positive and $\sup_{k \in \mathbb{N}_0}
\beta_k$ is finite, and we will obtain the optimal constants in
\eqref{e:Hoshiro} and \eqref{e:Hoshiroreverse} by taking $\phi(\rho)
= \frac{1}{2}\rho^2$ and $\tau = 2(1+s)$ (so that $\psi(\rho) =
\rho^{-s}$). In the subsequent section we give a very precise
description of these optimal constants; we delay our presentation of
this result because it is necessary to first establish some
technical notation. At this stage we emphasise that the case
$(d,\tau) = (4,2)$ is particularly special. Here, the sharp form of
\eqref{e:Hoshiro} and \eqref{e:Hoshiroreverse} is in fact an exact
\emph{identity}. By taking appropriate choices of $\phi$ and $\tau$,
we also obtain the analogous identities for solutions of the free
wave and Klein--Gordon equations. We collect these in the following.
\begin{theorem} \label{t:exactidentity}
On $\mathbb{R}^{4+1}$, all solutions of the Schr\"odinger equation
$i\partial_tu + \frac{1}{2}\Delta u = 0$ satisfy
\[
\int_{\mathbb{R}} \int_{\mathbb{R}^4} |(1-\Lambda)^{\frac{1}{4}}u(x,t)|^2 \, \frac{\mathrm{d}x\mathrm{d}t}{|x|^{2}}
= \pi \|u(0)\|_{L^2(\mathbb{R}^4)}^2,
\]
all solutions of the wave equation $\partial_{tt}u - \Delta u = 0$
satisfy
\begin{equation*}
2\int_\mathbb{R} \int_{\mathbb{R}^4} |(1-\Lambda)^{\frac{1}{4}}u(x,t)|^2 \, \frac{\mathrm{d}x\mathrm{d}t}{|x|^{2}} =
  \pi \left(\|u(0)\|_{\dot{H}^{\frac{1}{2}}(\mathbb{R}^4)}^2 + \|\partial_tu(0)\|_{\dot{H}^{-\frac{1}{2}}(\mathbb{R}^4)}^2\right),
\end{equation*}
and all solutions of the Klein--Gordon equation $\partial_{tt} u -
\Delta u + u = 0$ satisfy
\[
2\int_{\mathbb{R}} \int_{\mathbb{R}^4} |(1-\Delta)^{\frac{1}{4}}(1-\Lambda)^{\frac{1}{4}}u(x,t)|^2 \, \frac{\mathrm{d}x\mathrm{d}t}{|x|^{2}}
= \pi \left(\|u(0)\|_{L^2(\mathbb{R}^4)}^2 +
\|\nabla u(0)\|_{L^2(\mathbb{R}^4)}^2 +
\|\partial_tu(0)\|_{L^2(\mathbb{R}^4)}^2\right).
\]
\end{theorem}

In the related case where $\theta(\rho) = \rho^{\frac{\tau -
1}{4}}$, the situation is different because $\inf_{k \in
\mathbb{N}_0} \beta_k = 0$ and there is no reverse estimate. We do,
however, provide a explicit description of the upper bound $\sup_{k
\in \mathbb{N}_0} \beta_k$ in Section \ref{section:homogeneousg}. We
remark in passing that $\sup_{k \in \mathbb{N}_0} \beta_k$ is always
finite whenever $\theta(\rho) = O(\rho^{\frac{\tau - 1}{4}})$; this
is true because
\[
\frac{\Gamma(k+\tfrac{d-\tau}{2})}{\Gamma(k+\tfrac{d+\tau}{2}-1)} = O(k^{1-\tau}) \qquad \text{as $k \to \infty$,}
\]
which easily follows from Stirling's formula.

The case where $\theta$ is identically equal to one with no
smoothing along the sphere corresponds to \eqref{e:KY}. The optimal
constant for the forward estimate in \eqref{e:KY} has appeared in a
number of earlier works, including \cite{Chen}, \cite{OzawaRogers},
\cite{Watanabe} and our own \cite{BS} in the general case, and
\cite{Simon} for the case $s=0$ (see also \cite{WaltherBest}). In
\cite{BS}, we proceed using spectral considerations; in this work,
we build on \cite{BS} and the proof of Theorem \ref{t:Main} is also
based on spectral considerations. In \cite{OzawaRogers}, sharp angular refinements
of \eqref{e:KY} of a different nature to those considered in this paper
are established, in the forward direction,
by a different approach
through the sharp Hardy--Littlewood--Sobolev inequality on the sphere.
We also remark that when $\theta$
is identically equal to one, the sequence $(\beta_k)_{k \in
\mathbb{N}_0}$ is decreasing and tends to zero as $k$ tends to
infinity (see \cite{BS}). Hence there is no reverse inequality in
this case.

\begin{overview}
The upper bound in Theorem \ref{t:Main} is stated in the case of the
Schr\"odinger propapator in \cite{BS2}, and the substantially more
complete results of this paper were partially announced in
\cite{BS2}.

In Section \ref{section:explicit} we consider the important case
$\theta(\rho) = (1+\rho)^{\frac{\tau-1}{4}}$, where we provide a
comprehensive description of the optimal constants $\inf_{k \in
\mathbb{N}_0} \beta_k$ and $\sup_{k \in \mathbb{N}_0} \beta_k$ and
when these extrema are attained. The proofs of these results are
contained in Section \ref{section:computation}. As applications, we
provide the optimal constants and characterise the extremisers for
\eqref{e:Hoshiro} and \eqref{e:Hoshiroreverse}, along with analogous
results for the wave and Klein--Gordon equations. In Section \ref{section:TheoremMainProof} we prove Theorem
\ref{t:Main}.

Finally, in Section \ref{section:homogeneousg}, we provide several
further results, including an analysis of the case $\theta(\rho) = \rho^{\frac{\tau-1}{4}}$.
We also include some further generalisations of Theorem \ref{t:Main} to allow
weights which are not homogeneous, and dispersion relations and smoothing functions
$\phi$ and $\psi$ which are not required to satisfy \eqref{e:phipsi}. The disadvantage of
working in such generality is that a completely explicit description of optimal constants and
extremisers is not possible. The main focus of this paper is to establish such information
and this is the reason that we have presented the results in the Introduction in the
case where the weight is homogeneous, and $\phi$ and $\psi$ satisfy \eqref{e:phipsi}.
\end{overview}

\section{The case $\theta(\rho) = (1+\rho)^{\frac{\tau-1}{4}}$} \label{section:explicit}
\label{explicit} Theorem \ref{t:Main} makes it clear that to obtain
explicit expressions for the optimal constants in estimates of the
form \eqref{e:Main}, and to characterise the space of extremisers,
we must compute
\[
\mathbf{b}(d,\tau;\theta) = \inf_{k \in \mathbb{N}_0} \beta_k \qquad \text{and} \qquad \mathbf{B}(d,\tau;\theta) = \sup_{k \in \mathbb{N}_0} \beta_k\,,
\]
and understand the index sets
\[
\mathbf{k}(d,\tau;\theta) = \{ k \in \mathbb{N}_0 : \beta_{k} = \mathbf{b}(d,\tau;\theta) \} \qquad \text{and}
\qquad \mathbf{K}(d,\tau;\theta) = \{ k \in \mathbb{N}_0 : \beta_{k} = \mathbf{B}(d,\tau;\theta)\},
\]
where $\beta_k$ is given by
\[
\beta_k = \beta_k(d,\tau;\theta) = \pi 2^{2-\tau} \frac{\Gamma(\tau - 1)\Gamma(k+\tfrac{d-\tau}{2})}
{\Gamma(\tfrac{\tau}{2})^2\Gamma(k+\tfrac{d+\tau}{2}-1)} \,|\theta(k(k+d-2))|^2.
\]
The main result in this section is to do this in the case
\[
\theta(\rho) = (1+\rho)^{\frac{\tau - 1}{4}}.
\]
In order to state our result here, it is necessary to introduce a
little notation. For $d \geq 5$, we introduce two parameters $\tau_*
\in (1,d)$ and $\tau^* \in (1,d)$, depending only on $d$, as the
unique solution of the equations
\begin{equation*}
d^{\frac{\tau_*-1}{2}}\bigg(\frac{d-\tau_*}{2}\bigg) = \frac{d+\tau_*}{2}-1
\end{equation*}
and
\[
\Gamma\bigg(\frac{d-\tau^*}{2}\bigg) = \Gamma\bigg(\frac{d+\tau^*}{2} - 1\bigg)
\]
respectively. It is not immediately clear that $\tau_*$ and $\tau^*$
are well-defined so we provide a proof of this at the end of this
section, in order not to delay the presentation of the main results
of this section.
\begin{remark}
We will show in the course of the proof of the following theorem
that $\tau_* \leq \tau^*$.
\end{remark}
For $d \geq 5$ and $\tau \in [\tau_*,d)$ we let $k(\tau)$ be the unique
non-negative real number such that
\begin{equation} \label{e:k(a)defn}
\frac{2k(\tau)+d-\tau}{2k(\tau)+d+\tau-2}\bigg(\frac{1+(k(\tau)+1)(k(\tau)+d-1)}{1+k(\tau)(k(\tau)+d-2)}\bigg)^{\frac{\tau-1}{2}} = 1
\end{equation}
and let $k^*(\tau)$ denote the smallest integer greater than or
equal to $k(\tau)$. We show that $k(\tau)$ is well-defined during
the proof of the following (in Section \ref{section:computation}).
\begin{theorem} \label{t:betak}
Let $d \geq 2$, $\tau \in (1,d)$ and $\theta(\rho) =
(1+\rho)^{\frac{\tau - 1}{4}}$. Then the constants
$\mathbf{b}(d,\tau;\theta)$ and $\mathbf{B}(d,\tau;\theta)$ are
given by
\begin{table}[ht]
\begin{tabular}{ | c | c | c | }
\hline
$(d,\tau)$ & $\mathbf{b}(d,\tau;\theta)$ & $\mathbf{B}(d,\tau;\theta)$ \\ \hline \hline
$d=2,3$   & $\lim_{k \to \infty} \beta_k$ & $\beta_0$ \\ \hline
$d=4,\tau\in(1,2)$  & $\beta_0$ & $\lim_{k \to \infty} \beta_k$ \\ \hline
$d=4, \tau = 2$ & $\pi$ & $\pi$ \\ \hline
$d=4, \tau \in (2,4)$ & $\lim_{k \to \infty} \beta_k$ & $\beta_0$ \\ \hline
$d=5, \tau \in (1,\tau_*)$ & $\beta_0$ & $\lim_{k \to \infty} \beta_k$ \\ \hline
$d=5, \tau \in [\tau_*,\tau^*)$ & $\beta_{k^*(\tau)}$ & $\lim_{k \to \infty} \beta_k$ \\ \hline
$d=5, \tau \in [\tau^*,5)$ & $\beta_{k^*(\tau)}$ & $\beta_0$ \\ \hline
$d \geq 6, \tau \in (1,\tau_*)$ & $\beta_0$ & $\lim_{k \to \infty} \beta_k$ \\ \hline
$d \geq 6, \tau \in [\tau_*,\tau^*)$ & $\beta_1$  & $\lim_{k \to \infty} \beta_k$ \\ \hline
$d \geq 6, \tau \in [\tau^*,d)$ & $\beta_1$ & $\beta_0$ \\ \hline
\end{tabular}
\end{table}
\newline and the index sets $\mathbf{k}(d,\tau;\theta)$ and $\mathbf{K}(d,\tau;\theta)$ are
given by
\end{theorem}

\begin{table}[ht]
\begin{tabular}{ | c | c | c | }
\hline
$(d,\tau)$ & $\mathbf{k}(d,\tau;\theta)$ & $\mathbf{K}(d,\tau;\theta)$ \\ \hline \hline
$d=2,3$   & $\emptyset$ & \{0\} \\ \hline
$d=4,\tau\in(1,2)$  & \{0\} & $\emptyset$ \\ \hline
$d=4, \tau = 2$ & $\mathbb{N}_0$ & $\mathbb{N}_0$ \\ \hline
$d=4, \tau \in (2,4)$ & $\emptyset$ & \{0\} \\ \hline
$d=5, \tau \in (1,\tau_*)$ & \{0\} & $\emptyset$ \\ \hline
$d=5, \tau \in [\tau_*,\tau^*), k(\tau) \notin \mathbb{N}_0$ & $\{k^*(\tau)\}$ & $\emptyset$ \\ \hline
$d=5, \tau \in [\tau_*,\tau^*), k(\tau)  \in \mathbb{N}_0$ & $\{k^*(\tau),k^*(\tau)+1\}$ & $\emptyset$ \\ \hline
$d=5, \tau \in [\tau^*,5), k(\tau) \notin \mathbb{N}_0$ & $\{k^*(\tau)\}$ & \{0\} \\ \hline
$d=5, \tau \in [\tau^*,5), k(\tau)  \in \mathbb{N}_0$ & $\{k^*(\tau),k^*(\tau)+1\}$ & \{0\} \\ \hline
$d \geq 6, \tau \in (1,\tau_*)$ & \{0\} & $\emptyset$ \\ \hline
$d \geq 6, \tau = \tau_*$ & \{0,1\} & $\emptyset$ \\ \hline
$d \geq 6, \tau \in (\tau_*,\tau^*)$ & \{1\}  & $\emptyset$ \\ \hline
$d \geq 6, \tau \in [\tau^*,d)$ & \{1\} & \{0\} \\ \hline
\end{tabular}
\end{table}

We can combine Theorems \ref{t:Main} and \ref{t:betak} and give a
precise description of the optimal constants and extremisers in
\eqref{e:Hoshiro} and \eqref{e:Hoshiroreverse}.
\begin{notation}
For $d \geq 2$ and $s \in (-\frac{1}{2},\frac{d}{2}-1)$, define
constants $\mathbf{c}(d,s)$ and $\mathbf{C}(d,s)$ by
\begin{equation*}
\mathbf{c}(d,s) = \inf_{k \in \mathbb{N}_0} \pi 2^{-2s} \frac{\Gamma(1+2s)\Gamma(k+\tfrac{d-2}{2}-s)}
{\Gamma(1+s)^2\Gamma(k+\tfrac{d}{2}+s)} \,(1+k(k+d-2))^{s+\frac{1}{2}}
\end{equation*}
and
\begin{equation*}
\mathbf{C}(d,s) = \sup_{k \in \mathbb{N}_0} \pi 2^{-2s} \frac{\Gamma(1+2s)\Gamma(k+\tfrac{d-2}{2}-s)}
{\Gamma(1+s)^2\Gamma(k+\tfrac{d}{2}+s)} \,(1+k(k+d-2))^{s+\frac{1}{2}}.
\end{equation*}
\end{notation}
Observe that
\begin{equation} \label{e:cblink}
\mathbf{c}(d,s) = \mathbf{b}(d,2(1+s);\theta)
\end{equation}
and
\begin{equation} \label{e:CBlink}
\mathbf{C}(d,s) = \mathbf{B}(d,2(1+s);\theta)
\end{equation}
where $\theta$ is given by
\begin{equation*}
  \theta(\rho) = (1+\rho)^{\frac{1+2s}{4}}.
\end{equation*}
\begin{corollary} \label{c:S}
Let $d \geq 2$, $s \in (-\frac{1}{2},\frac{d}{2}-1)$ and suppose that $i\partial_t u + \frac{1}{2}\Delta u = 0$ on $\mathbb{R}^{d+1}$. Then
\[
\mathbf{c}(d,s) \|u(0)\|_{\dot{H}^s(\mathbb{R}^d)}^2 \leq \int_{\mathbb{R}} \int_{\mathbb{R}^d}
|(1-\Lambda)^{\frac{1+2s}{4}}u(x,t)|^2 \, \frac{\mathrm{d}x\mathrm{d}t}{|x|^{2(1+s)}}  \leq \mathbf{C}(d,s) \|u(0)\|_{\dot{H}^s(\mathbb{R}^d)}^2,
\]
and the constants are optimal.
\end{corollary}
Likewise, for the wave and Klein--Gordon equations we have the
following.
\begin{corollary} \label{c:W}
Let $d \geq 2$, $s \in (0,\frac{d-1}{2})$ and suppose that
$\partial_{tt} u - \Delta u = 0$ on $\mathbb{R}^{d+1}$. Then
\begin{equation*}
 \mathbf{c}(d,s-\tfrac{1}{2}) \|(u(0),\partial_tu(0))\|^2 \leq
  2\int_\mathbb{R} \int_{\mathbb{R}^d} |(1-\Lambda)^{\frac{s}{2}}u(x,t)|^2 \, \frac{\mathrm{d}x\mathrm{d}t}{|x|^{1+2s}} \leq
  \mathbf{C}(d,s-\tfrac{1}{2}) \|(u(0),\partial_tu(0))\|^2
\end{equation*}
and the constants are optimal. Here, the norm on the initial data is
given by
\[
\|(u(0),\partial_tu(0))\|^2 = \|u(0)\|_{\dot{H}^s(\mathbb{R}^d)}^2 + \|\partial_tu(0)\|_{\dot{H}^{s-1}(\mathbb{R}^d)}^2.
\]
\end{corollary}
\begin{corollary} \label{c:KG}
Let $d \geq 2$, $s \in (-\frac{1}{2},\frac{d}{2}-1)$ and suppose
that $\partial_{tt} u - \Delta u + u = 0$ on $\mathbb{R}^{d+1}$.
Then
\[
\mathbf{c}(d,s) \|(u(0),\partial_tu(0))\|^2 \leq
2\int_{\mathbb{R}} \int_{\mathbb{R}^d} |(1-\Delta)^{\frac{1}{4}}(1-\Lambda)^{\frac{1+2s}{4}}u(x,t)|^2 \, \frac{\mathrm{d}x\mathrm{d}t}{|x|^{2(1+s)}}
\leq \mathbf{C}(d,s) \|(u(0),\partial_tu(0))\|^2
\]
and the constants are optimal. Here, the norm on the initial data is
given by
\[
\|(u(0),\partial_tu(0))\|^2 = \|u(0)\|_{\dot{H}^s(\mathbb{R}^d)}^2 +
\|u(0)\|_{\dot{H}^{s+1}(\mathbb{R}^d)}^2 +
\|\partial_tu(0)\|_{\dot{H}^{s}(\mathbb{R}^d)}^2.
\]
\end{corollary}

We remark that Theorem \ref{t:exactidentity} is a straightforward consequence
of Corollaries \ref{c:S}--\ref{c:KG} and Theorem \ref{t:betak} to obtain $C(4,0)=c(4,0)=\pi$.

\begin{proof}[Proof of Corollaries \ref{c:S}, \ref{c:W} and
\ref{c:KG}] Corollary \ref{c:S} follows immediately from Theorem
\ref{t:Main} by taking $\phi(\rho) = \frac{1}{2}\rho^2$, $\psi(\rho)
= \rho^{-s}$ and $\tau = 2(1+s)$, clearly satisfying
\eqref{e:phipsi}. For Corollary \ref{c:W}, we write the solution of
the wave equation $u$ as $u_+ + u_-$, where
\[
u_\pm(t) = \exp(\pm it|\nabla|)f_\pm
\]
and
\[
u(0) = f_+ + f_- \qquad \text{and} \qquad \partial_tu(0) = i|\nabla|(f_+ - f_-).
\]
Then
\[
\int_\mathbb{R} \int_{\mathbb{R}^d} |(1-\Lambda)^{\frac{s}{2}}u(x,t)|^2 \, \frac{\mathrm{d}x\mathrm{d}t}{|x|^{1+2s}} =
\int_\mathbb{R} \int_{\mathbb{R}^d} |(1-\Lambda)^{\frac{s}{2}}u_+(x,t)|^2 \, \frac{\mathrm{d}x\mathrm{d}t}{|x|^{1+2s}} +
\int_\mathbb{R} \int_{\mathbb{R}^d} |(1-\Lambda)^{\frac{s}{2}}u_-(x,t)|^2 \, \frac{\mathrm{d}x\mathrm{d}t}{|x|^{1+2s}}
\]
by Plancherel's theorem and the fact that the Fourier transforms in
time of $u_+$ and $u_-$ are disjoint. Corollary \ref{c:W} now
follows from two applications of Theorem \ref{t:Main}, with
$\phi(\rho) = \pm \rho$, $\psi(\rho) = \rho^{-s}$ and $\tau = 1+2s$,
and the parallelogram law. The proof of Corollary \ref{c:KG} is
similar, using $\phi(\rho) = \pm(1+\rho^2)^{1/2}$, $\psi(\rho) =
(1+\rho^2)^{-\frac{1}{4}}\rho^{-s}$ and $\tau = 2(1+s)$, and we omit
the details.
\end{proof}
Of course, Theorem \ref{t:betak} provides a precise description of
the optimal constants $\mathbf{c}(d,s)$ and $\mathbf{C}(d,s)$
appearing in Corollaries \ref{c:S}, \ref{c:W} and \ref{c:KG}
(through \eqref{e:cblink} and \eqref{e:CBlink}). The constants
$\beta_0$, $\beta_1$ and $\lim_{k \to \infty} \beta_k$ appearing in
Theorem \ref{t:betak} are given explicitly in terms of $d$ and
$\tau$ as follows
\begin{align}
\beta_0 & = \pi 2^{2-\tau} \frac{\Gamma(\tau - 1)\Gamma(\tfrac{d-\tau}{2})}
{\Gamma(\tfrac{\tau}{2})^2\Gamma(\tfrac{d+\tau}{2}-1)}  \label{e:alpha0} \\
\beta_1 & = \pi 2^{2-\tau} d^{\frac{\tau - 1}{2}} \frac{\Gamma(\tau - 1)\Gamma(1+\tfrac{d-\tau}{2})}
{\Gamma(\tfrac{\tau}{2})^2\Gamma(\tfrac{d+\tau}{2})}  \\
\lim_{k \to \infty} \beta_k & = \pi 2^{2-\tau} \frac{\Gamma(\tau - 1)}{\Gamma(\frac{\tau}{2})^2}\,, \label{e:alphaklimit}
\end{align}
where \eqref{e:alphaklimit} follows easily from Stirling's formula.
In the exceptional case $d=5$ and $\tau \in (\tau_*,5)$ the lower
bound $\mathbf{b}(5,\tau;\theta)$ is given in terms of $k(\tau)$
which is implicitly defined. We are, at least, able to provide the
following bounds on $k(\tau)$.
\begin{proposition} \label{p:growth}
Let $d=5$ and $\tau \in (\tau_*,5)$. The unique positive real number $k(\tau)$
for which \eqref{e:k(a)defn} holds satisfies the following bounds
\begin{equation*}
\frac{C_1}{(5-\tau)^{1/4}} \leq k(\tau) \leq \frac{C_2}{(5-\tau)^{1/2}}
\end{equation*}
for some positive constants $C_1$ and $C_2$.
\end{proposition}
We have not attempted to sharpen these bounds by bringing the
exponents $\frac{1}{4}$ and $\frac{1}{2}$ closer together, although
this would be an interesting problem to solve.

Additionally, Theorem \ref{t:betak} allows one to characterise the
space of extremising initial data in Corollaries \ref{c:S},
\ref{c:W} and \ref{c:KG}. For example, in spatial dimensions $d=2,
3$, and any $\tau \in (1,d)$, we know that the lower bounds do not
possess extremising initial data, and the upper bounds are realised
if and only if the initial data is radially symmetric.

Clearly, the case of five spatial dimensions is the most subtle in
Theorem \ref{t:betak}. Although we cannot provide a concrete
explanation for this, it is conceivable this is related to the
amusing fact that the volume of the unit sphere as a function of the
dimension has a global maximum in five dimensions.

As promised, we end this section with a justification that the
parameters $\tau_*$ and $\tau^*$ are well-defined.
\begin{proof}[Proof that $\tau_*$ is well-defined]
Recall that we are assuming $d \geq 5$. Observe that
\begin{equation*}
\Phi(\tau) := \frac{\partial}{\partial \tau} \bigg(d^{\frac{\tau-1}{2}} \frac{\frac{d-\tau}{2}}{\frac{d+\tau}{2}-1}\bigg)
=\frac{d^{\frac{\tau-1}{2}}}{2(d-2+\tau)^2}\left((\log d)(d-\tau)(d-2+\tau) - 4(d-1)\right)
\end{equation*}
has at most two roots. These roots are given by
\[
\tau = 1 \pm \sqrt{1 + d(d-2) - \frac{4(d-1)}{\log d}}
\]
and therefore at most one of these roots lies in $(1,d)$. Furthermore
\[
\Phi(1) = \frac{1}{2(d-1)}((\log d)(d-1)-4) > 0
\]
for $d \geq 5$, and
\[
\Phi(d) = -\frac{d^{\frac{d-1}{2}}}{2(d-1)} < 0
\]
and it follows that there is precisely one root of $\Phi$ in the
interval $(1,d)$. Clearly
$$
d^{\frac{\tau-1}{2}} \frac{\frac{d-\tau}{2}}{\frac{d+\tau}{2}-1} = \left\{\begin{array}{llll}
1 & \text{if $\tau = 1$} \\
0 & \text{if $\tau = d$}
\end{array}  \right.
$$
and it follows that $\tau_*$ exists and is unique.
\end{proof}

\begin{proof}[Proof that $\tau^*$ is well-defined]
Again, here we are only considering $d \geq 5$. Let
$$
\Upsilon(t) = \frac{\Gamma(t)}{\Gamma(d-1-t)}
$$
for $t \in (0,\frac{d-1}{2})$. Then, of course,
$$
\frac{\Gamma(\frac{d-\tau}{2})}{\Gamma(\frac{d+\tau}{2}-1)} = \Upsilon(t)
$$
if $t = \frac{d-\tau}{2}$ (and note that $t \in (0,\frac{d-1}{2})$
if and only if $\tau \in (1,d)$). So, it suffices to show that there
exists a unique $t^* \in (0,\frac{d-1}{2})$ such that $\Upsilon(t^*)
= 1$.

To this end, we observe that $\Upsilon$ is log-convex on
$(0,\frac{d-1}2)$ because
$$
(\log \Upsilon)'(t) = \psi(t) + \psi(d-1-t)
$$
where $\psi := (\log \Gamma)'$ is the digamma function. We note that
\begin{equation} \label{e:digammaprops}
\psi(t) = -\gamma-\frac{1}{t}+t \sum_{j = 1}^\infty \frac{1}{j(t+j)}
\qquad \text{and}
\qquad \psi'(t) = \sum_{j = 0}^\infty \frac{1}{(t+j)^2},
\end{equation}
where
$$
\gamma=\lim_{m\to\infty}\left\{\frac{1}{1}+\frac{1}{2}+\cdots+\frac{1}{m}
-\log m\right\}=0.5772157\ldots
$$
(see Whittaker--Watson \cite[Section 12.16]{WW}) and hence
 $\psi'$ is a decreasing function on $(0,\infty)$.
Then
$$
(\log \Upsilon)''(t) = \psi'(t) - \psi'(d-1-t) > 0
$$
because we have $t < \frac{d-1}{2}$, giving the claimed
log-convexity of $\Upsilon$.

So, in particular, $\Upsilon$ must be convex on $(0,\frac{d-1}{2})$. We
have $\lim_{t \to 0+} \Upsilon(t) = +\infty$ and at the other
endpoint, we have $\Upsilon(\frac{d-1}{2}) = 1$. Also,
$$
\Upsilon'(t) = \Upsilon(t) (\psi(t) + \psi(d-1-t)).
$$
It can be shown from \eqref{e:digammaprops} that $\psi(\frac{d-1}{2}) > 0$ 
for $d \geq 5$ and therefore $\Upsilon'(\frac{d-1}{2}) =
2\Upsilon(\frac{d-1}{2})\psi(\frac{d-1}{2}) > 0$ for $d \geq 5$. From this
we know that $\Upsilon(t)$ is increasing for $t$ sufficiently close to
$\frac{d-1}{2}$. By the Intermediate Value Theorem, there exists
$t^* \in (0,\frac{d-1}{2})$ such that $\Upsilon(t^*) = 1$. This must
be unique since $\Upsilon$ is convex on $(0,\frac{d-1}{2})$.
\end{proof}

\section{Proof of Theorem \ref{t:Main}}\label{section:TheoremMainProof}

First we will need to provide a brief discussion of spherical
harmonics. Essentially, the arguments in this section are already
present in \cite{BS2}. We include the details here for
self-containedness, and to clarify a small technical point in the
expression of the projection $H_k$ from $L^2(\mathbb{S}^{d-1})$ to
$\mathcal{H}_k$; here we use Legendre polynomials instead of
Gegenbauer polynomials to include the case $d = 2$ in a more
transparent way.

Let $\mathcal{A}_k$ be the space of solid spherical harmonics of
degree $k$ (these are harmonic polynomials on $\mathbb{R}^d$ which
are homogeneous of degree $k$), and let $\mathcal{H}_k$ be the space
of spherical harmonics of degree $k$ (these are restriction of
functions in $\mathcal{A}_k$ to the sphere $\mathbb{S}^{d-1}$). Then
the eigenvalues of the Laplace--Beltrami operator $-\Lambda$ on the
sphere $\mathbb{S}^{d-1}$ are
\begin{equation}\label{def:mu}
\mu_k=
k(k+d-2)   
\end{equation}
and the corresponding eigenspaces are $\mathcal{H}_k$. The
projection $H_k$ from $L^2(\mathbb{S}^{d-1})$ to $\mathcal{H}_k$ can
be written
\[
H_kf(\omega)=\frac{N_{k,d}}{|\mathbb{S}^{d-1}|}
\int_{\mathbb{S}^{d-1}}P_{k,d}(\omega\cdot\widetilde{\omega})f(\widetilde{\omega})
\,\mathrm{d}\widetilde{\omega},
\]
where
\[
N_{k,d}=\frac{(2k+d-2)(k+d-3)!}{k!(d-2)!},
\]
$|\mathbb{S}^{d-1}|$ is the surface area of the sphere and $P_{k,d}$
is the Legendre polynomial of degree $k$ (see \cite{AtkinsonHan}).

Recall that we use the notation $\mathfrak{H}_k$ for the space of
all linear combinations of functions
\begin{equation*}
\xi \mapsto P(\xi) f_0(|\xi|) |\xi|^{-d/2-k+1/2}
\end{equation*}
where $P \in \mathcal{A}_k$ and $f_0 \in L^2(0,\infty)$. These
spaces allow us to decompose $L^2(\mathbb{R}^d)$ as
\begin{equation*} \label{e:directsum}
L^2(\mathbb{R}^d) = \bigoplus_{k=0}^\infty \mathfrak{H}_k\,,
\end{equation*}
where this is a complete orthogonal direct sum decomposition in the
sense that the closed subspaces $\mathfrak{H}_k$ are mutually
orthogonal in $L^2(\mathbb{R}^d)$ for $k \in \mathbb{N}_0$, and
every $f \in L^2(\mathbb{R}^d)$ can be written $f =
\sum_{k=0}^\infty f_k$ for some $f_k \in \mathfrak{H}_k$. We refer
the reader to \cite{Shimakura} and \cite{SteinWeiss} for further
details.
\par
The operator $H_k$ on $L^2(\mathbb{S}^{d-1})$ can be homogeneously
extended to $L^2(\mathbb{R}^d)$ in a natural way by
\[
H_kf(x)=\frac{N_{k,d}}{|\mathbb{S}^{d-1}|}
\int_{\mathbb{S}^{d-1}}P_{k,d}(x'\cdot\widetilde{\omega})f(|x|\widetilde{\omega})
\,\mathrm{d}\widetilde{\omega},
\]
where $x' = |x|^{-1}x$, and we shall use the same notation $H_k$ as long as there is no
confusion. In this way, the Laplace--Beltrami operator $-\Lambda$
can be also regarded as an operator on $L^2(\mathbb{R}^d)$ by using
the spectral decomposition
\[
-\Lambda=\sum_{k=0}^\infty \mu_k H_k.
\]
It is easy to see that the eigenvalues of this operator are again
$\{\mu_k\}_{k=0}^\infty$, and $H_k$ is the projection to the
eigenspace $\mathfrak{H}_k$ of $\mu_k$ for each $k\in\N_0$. For
any functions $\theta(\rho)$ of $\rho\in[0,\infty)$, we can
also define $\theta(-\Lambda)$ as an operator on $L^2(\mathbb{R}^d)$
by
\[
\theta(-\Lambda)=\sum_{k=0}^\infty \theta(\mu_k) H_k.
\]
\begin{proposition} \label{p:commute}
For each $k \in \mathbb{N}_0$, the operator $H_k$ commutes with the Fourier transform and the inverse Fourier transform. In particular, each subspace $\mathfrak{H}_k$ is invariant under the action of these operators.
\end{proposition}
Since we are handling explicit constants, we clarify that $\widehat{f}$ is the Fourier transform of $f$ given by
\begin{equation*}
\widehat{f}(\xi) = \int_{\mathbb{R}^d} f(x) \exp(-i x \cdot \xi) \,\mathrm{d}x.
\end{equation*}
\begin{proof}[Proof of Proposition \ref{p:commute}]
Using polar coordinates, we have
\[
\widehat{H_kf}(x) = \frac{N_{k,d}}{|\mathbb{S}^{d-1}|} \int_{0}^\infty \int_{\mathbb{S}^{d-1}} \int_{\mathbb{S}^{d-1}} \exp(-ir\omega \cdot x) P_{k,d}(\omega \cdot \widetilde{\omega}) f(r\widetilde{\omega}) r^{d-1} \, \mathrm{d}\omega \mathrm{d}\widetilde{\omega} \mathrm{d}r
\]
and
\[
H_k \widehat{f}(x) = \frac{N_{k,d}}{|\mathbb{S}^{d-1}|} \int_{0}^\infty \int_{\mathbb{S}^{d-1}} \int_{\mathbb{S}^{d-1}} \exp(-i |x| \widetilde{\omega} \cdot r \omega) P_{k,d}(x' \cdot \widetilde{\omega}) f(r\omega) r^{d-1} \, \mathrm{d}\omega \mathrm{d}\widetilde{\omega} \mathrm{d}r\,
\]
so it suffices to check that
\[
\int_{\mathbb{S}^{d-1}} \int_{\mathbb{S}^{d-1}} \exp(-ir\omega \cdot x) P_{k,d}(\omega \cdot \widetilde{\omega}) f(r\widetilde{\omega})\, \mathrm{d}\omega \mathrm{d}\widetilde{\omega} = \int_{\mathbb{S}^{d-1}} \int_{\mathbb{S}^{d-1}} \exp(-i |x| \widetilde{\omega} \cdot r \omega) P_{k,d}(x' \cdot \widetilde{\omega}) f(r\omega) \, \mathrm{d}\omega \mathrm{d}\widetilde{\omega} 
\]
for each $x \in \mathbb{R}^d$ and $r > 0$. By switching the $\omega$ and $\widetilde{\omega}$ variables on the left-hand side, it now suffices to show
\begin{equation} \label{e:commute}
\int_{\mathbb{S}^{d-1}} \exp(-ir|x| \widetilde{\omega} \cdot x') P_{k,d}(\widetilde{\omega} \cdot \omega)  \mathrm{d}\widetilde{\omega} = \int_{\mathbb{S}^{d-1}} \exp(-ir |x| \widetilde{\omega} \cdot \omega) P_{k,d}(x' \cdot \widetilde{\omega}) \, \mathrm{d}\widetilde{\omega}
\end{equation}
for each $x \in \mathbb{R}^d$, $r > 0$ and $\omega \in \mathbb{S}^{d-1}$. However, $\widetilde{\omega}
\mapsto P_{k,d}(\widetilde{\omega} \cdot \omega)$ and $\widetilde{\omega} \mapsto
P_{k,d}(x' \cdot \widetilde{\omega})$ are spherical
harmonics of degree $k$ and we may apply by the Funk--Hecke theorem
(see, for example, \cite{AtkinsonHan}) to see that both sides of
\eqref{e:commute} are equal to 
\begin{equation*}
|\mathbb{S}^{d-2}|\, P_{k,d}(x'\cdot\omega) \int^1_{-1}P_{k,d}(s)\exp(-irs|x|)(1-s^2)^{\frac{d-3}2}\,\mathrm{d}s \,,
\end{equation*}
which gives the desired claim. The proof for the inverse Fourier transform is almost identical and we omit the details.
\end{proof}
It follows from Proposition \ref{p:commute} that 
\[
(H_k\Psi(|\nabla|)f)\widehat{\,}\,(\xi) = H_k (\Psi(|\nabla|)f)\widehat{\,}\,(\xi) = \Psi(|\xi|) H_k\widehat{f}(\xi) = (\Psi(|\nabla| H_k f)\widehat{\,}\, (\xi)
\]
and therefore $H_k\Psi(|\nabla|) = \Psi(|\nabla|)H_k$. From this, we also know that $\theta(-\Lambda)$ also commutes with the Fourier transform, its inverse, and $\Psi(|\nabla|)$. We use this observation in order to prove Theorem \ref{t:Main}.

\begin{proof}[Proof of Theorem \ref{t:Main}]
Let $S_\theta : L^2(\mathbb{R}^d) \to L^2(\mathbb{R}^{d+1})$ be the
linear operator given by
\begin{equation} \label{e:Sdefn}
S_\theta f(x,t) = |x|^{-\frac{\tau}{2}} \theta(-\Lambda) \int_{\mathbb{R}^d} \exp(i(x \cdot \xi +
t\phi(|\xi|)) \psi(|\xi|) f(\xi) \, \mathrm{d}\xi
\end{equation}
for Schwartz functions $f : \mathbb{R}^d \to \mathbb{C}$, $(x,t) \in
\mathbb{R}^d \times \mathbb{R}$ and where $\theta(-\Lambda)$ is an
operation in the $x$-variable. The relevance of the operator
$S_\theta$ is seen through the expression
\begin{equation} \label{e:Slink}
|x|^{-\frac{\tau}{2}} \theta(-\Lambda) \psi(|\nabla|)\exp(it\phi(|\nabla|))f(x)
= (2\pi)^{-d} S_\theta\widehat{f}(x,t).
\end{equation}

\begin{proposition}\label{t:Sg}
Let $\tau \in (1,d)$. Then the operator $S_\theta^*S_\theta$ has the
spectral decomposition
\[
S_\theta^* S_\theta
=\sum_{k=0}^\infty \lambda_k\,|\theta(\mu_k)|^2\,H_k,
\]
where, for each $k \in \mathbb{N}_0$,
\begin{equation}\label{def:lambda}
\lambda_k=
(2\pi)^{d+1}2^{1-\tau}
\frac{\Gamma(\tau-1)\Gamma(k+\tfrac{d-\tau}{2})}
{\Gamma(\frac{\tau}{2})^2\Gamma(k+\tfrac{d+\tau}{2}-1)}
\end{equation}
and $\mu_k = k(k+d-2)$.
\end{proposition}
\begin{proof}
When $\theta$ is identically equal to one, this follows from
\cite{BS} (see Theorem 1.5). The general case follows from 
\begin{equation} \label{e:Sg}
S_\theta=S_1 \circ \theta(-\Lambda).
\end{equation}
To see \eqref{e:Sg}, we simply use our observation that $\theta(-\Lambda)$ commutes with the inverse Fourier transform and, for each fixed $t$, commutes with the operator $\psi(|\nabla|) \exp(it\phi(|\nabla|))$. 
\end{proof}
Clearly, from Proposition \ref{t:Sg} we have
\begin{align*}
\|S_\theta f\|^2_{L^2(\mathbb{R}^{d+1})} = \sum_{k=0}^\infty\sum_{\ell=0}^\infty\lambda_k\,|\theta(\mu_k)|^2
(H_kf,H_\ell f)_{L^2(\mathbb{R}^d)}
=\sum_{k=0}^\infty\lambda_k\,|\theta(\mu_k)|^2\|H_kf\|^2_{L^2(\mathbb{R}^d)}
\end{align*}
and therefore
\begin{equation*}
\inf_{k \in \mathbb{N}_0} \lambda_k\,|\theta(\mu_k)|^2 \,\|f\|_{L^2(\mathbb{R}^d)}^2 \leq
\|S_\theta f\|^2_{L^2(\mathbb{R}^{d+1})} \leq \sup_{k \in \mathbb{N}_0} \lambda_k\,|\theta(\mu_k)|^2 \,\|f\|_{L^2(\mathbb{R}^d)}^2.
\end{equation*}
Using \eqref{e:Slink} and Plancherel's theorem
$\|\widehat{f}\|_{L^2(\R^d)}^2 = (2\pi)^{d}\|f\|_{L^2(\R^d)}^2$ we
obtain
\begin{equation*}
\inf_{k \in \mathbb{N}_0} \beta_k \, \|f\|_2^2 \leq \int_{\mathbb{R}} \int_{\mathbb{R}^d}
|\psi(|\nabla|) \, \theta(-\Lambda) \, \exp(it\phi(|\nabla|))f(x)|^2 \, \frac{\mathrm{d}x\mathrm{d}t}{|x|^\tau} \leq \sup_{k \in \mathbb{N}_0} \beta_k \, \|f\|_2^2,
\end{equation*}
where the $\beta_k$ are as given in the statement of Theorem
\ref{t:Main}. The optimality of the constants and the remaining
claims concerning extremisers follow in a straightforward way using
the fact that
\[\|S_\theta f\|^2_{L^2(\mathbb{R}^{d+1})} = \lambda_k\,|\theta(\mu_k)|^2\|f\|^2_{L^2(\mathbb{R}^d)}
\]
for any $f\in \mathfrak{H}_{k}\setminus\{0\}$ and any $k\in\N_0$, orthogonality arguments and Proposition \ref{p:commute}.
\end{proof}

\section{Proofs of Theorem \ref{t:betak} and Proposition \ref{p:growth}}\label{section:computation}
Recall that $\mathbf{b}(d,\tau;\theta) = \inf_{k \in \mathbb{N}_0}
\beta_k$ and $\mathbf{B}(d,\tau;\theta) = \sup_{k \in \mathbb{N}_0}
\beta_k$ where
\[
\beta_k = \beta_k(d,\tau;\theta) = \pi 2^{2-\tau} \frac{\Gamma(\tau - 1)\Gamma(k+\tfrac{d-\tau}{2})}
{\Gamma(\tfrac{\tau}{2})^2\Gamma(k+\tfrac{d+\tau}{2}-1)} \,(1+k(k+d-2))^{\frac{\tau-1}{2}}
\]
for $\theta(\rho) = (1+\rho)^{\frac{\tau - 1}{4}}$.

\begin{proof}[Proof of Theorem \ref{t:betak}]
We set
\[
h(k,\tau) := \frac{\beta_{k+1}}{\beta_k}
=
\frac{2k+d-\tau}{2k+d+\tau-2}\,
\left(\frac{1+(k+1)(k+d-1)}{1+k(k+d-2)}\right)^{\frac{\tau-1}{2}}.
\]
Then we have that $h(k,\tau)\to1$ as $k\to\infty$ and will often use this fact
without notification.
Also we have
\[
\frac{\partial h}{\partial k}(k,\tau)=-A(d,k,\tau)\b{B_0(d,\tau)+B_1(d,\tau)k+B_2(d,\tau)k^2},
\]
where
\begin{align*}
A(d,k,\tau) & =
\frac{\tau - 1}
{2(2k+d+\tau-2)^2 \{1+k(k+d-2)\}^{2}}
\p{\frac{1+(k+1)(k+d-1)}{1+k(k+d-2)}}^{\frac{\tau-3}{2}} \\
B_0(d,\tau) & = d\big( (d-3)\tau(2-\tau) + (d-4)(d^2-d+2) \big), \\
B_1(d,\tau) & = 2(d-1)\big(\tau(2-\tau) + 3d(d-4)\big)= (d-1)B_2(d,\tau), \\
B_2(d,\tau) & = 2\big(\tau(2-\tau) + 3d(d-4)\big).
\end{align*}
Regarding the function $A$, the only property that we use in the
rest of the proof is that $A > 0$ since $\tau > 1$.

The cases where $d=2$, $d=3$ and $d=4$ with $\tau \in (2,4)$ are
straightforward to handle because, for such $(d,\tau)$, we have
$B_j(d,\tau) < 0$ for each $j=1,2,3$. This is clear from the
expressions:
\begin{table}[ht]
\begin{tabular}{ | c | c | c | c | }
\hline
$(d,\tau)$ & $B_0(d,\tau)$ & $B_1(d,\tau)$ & $B_2(d,\tau)$ \\ \hline \hline
$d=2$   & $-2\tau(2-\tau)-16$ & $2\tau(2-\tau)-24$ & $2\tau(2-\tau)-24$ \\ \hline
$d=3$  & $-24$ & $4\tau(2-\tau)-36$ & $2\tau(2-\tau) -18$ \\ \hline
$d=4$  & $4\tau(2-\tau)$ & $6\tau(2-\tau)$ & $2\tau(2-\tau)$ \\ \hline
\end{tabular}
\end{table}

This means $h(\cdot,\tau)$ is strictly increasing (to 1) so that
$(\beta_k)_{k \in \mathbb{N}_0}$ is a strictly decreasing sequence.
This implies
$$
\inf_{k \in \mathbb{N}_0} \beta_k = \lim_{k \to \infty} \beta_k \qquad \text{and} \qquad \sup_{k \in \mathbb{N}_0} \beta_k = \beta_0.
$$
In the special case $(d,\tau)=(4,2)$, we clearly have $B_j(4,2) = 0$
for each $j=1,2,3$. This means $h(\cdot,\tau) = 1$ and the sequence
$(\beta_k)_{k \in \mathbb{N}_0}$ is constant. Using, for example,
\eqref{e:alpha0}, we have that this constant value is equal to
$\pi$. Also, when $d=4$ and $\tau \in (1,2)$ it is clear that
$B_j(4,\tau) > 0$ for each $j=1,2,3$. In such a case,
$h(\cdot,\tau)$ is strictly decreasing, $(\beta_k)_{k \in
\mathbb{N}_0}$ is a strictly increasing sequence, and therefore
$$
\inf_{k \in \mathbb{N}_0} \beta_k = \beta_0 \qquad \text{and} \qquad \sup_{k \in \mathbb{N}_0} \beta_k = \lim_{k \to \infty} \beta_k.
$$

We next consider the case $d \geq 5$ and begin with some preliminary
observations. Recall that $\tau_*$ is uniquely defined by
$h(0,\tau_*) = 1$. It is also true that, for $\tau \in (1,d)$, we
have $h(0,\tau) < 1$ if and only if $\tau \in (\tau_*,d)$.

Since
\begin{equation*}
  \beta_0 = \pi 2^{2-\tau} \frac{\Gamma(\tau - 1)}
{\Gamma(\tfrac{\tau}{2})^2} \frac{\Gamma(\tfrac{d-\tau}{2})}{\Gamma(\tfrac{d+\tau}{2}-1)} =
\frac{\Gamma(\tfrac{d-\tau}{2})}{\Gamma(\tfrac{d+\tau}{2}-1)} \lim_{k \to \infty} \beta_k
\end{equation*}
we see that $\beta_0 = \lim_{k \to \infty} \beta_k$ when $\tau =
\tau^*$, and it is also true that, for $\tau \in (1,d)$, we have
\begin{equation} \label{e:asubstar}
\beta_0 > \lim_{k\to\infty} \beta_k \qquad \Longleftrightarrow \qquad \tau \in (\tau^*,d).
\end{equation}
We also record the following lemma, which is completely elementary.
\begin{lemma} \label{l:h}
Let $d \geq 5$ and $\tau \in (1,d)$. Then
 $h(k,\tau)$ is strictly decreasing to $1$ for sufficiently
large $k$. Furthermore, $h(\cdot,\tau)$ has at
most one stationary point on $[0,\infty)$, and when it exists it is
a global maximum on this domain.
\end{lemma}
\begin{proof}
Since $B_1(d,d) = 4d(d-1)(d-5)$ and $B_2(d,d) = 4d(d-5)$, it is
clear that $B_1(d,\tau), B_2(d,\tau) > 0$ for all $\tau \in (1,d)$.
This, of course, means that the quadratic function
$$
k \mapsto B_0(d,\tau) + B_1(d,\tau)k + B_2(d,\tau)k^2
$$
has at most one root on $[0,\infty)$. Since $A > 0$, it follows that
$h(\cdot,\tau)$ has at most one stationary point on $[0,\infty)$. It
is clear that $h(k,\tau)$ is strictly decreasing for sufficiently
large $k$ and therefore this stationary point must be a global
maximum when it exists.
\end{proof}
We can use Lemma \ref{l:h} to argue that if $\tau \in (1,\tau_*)$
then $h(0,\tau) > 1$ and therefore $h(k,\tau) > 1$ for all $k \geq
0$. Consequently, $(\beta_k)_{k \in \mathbb{N}_0}$ is strictly
increasing so that
\begin{equation} \label{e:abigastar}
\inf_{k \in \mathbb{N}_0} \beta_k = \beta_0 \qquad \text{and} \qquad \sup_{k \in \mathbb{N}_0} \beta_k = \lim_{k \to \infty} \beta_k.
\end{equation}
\begin{remark}
We can now deduce that $\tau_*$ cannot exceed $\tau^*$. If it were
true that $\tau^* < \tau_*$ then for $\tau \in (\tau^*,\tau_*)$ we
know from \eqref{e:abigastar} that $\lim_{k \to \infty} \beta_k >
\beta_0$. However, this contradicts \eqref{e:asubstar}.
\end{remark}

If $\tau \in [\tau_*,d)$, so that $h(0,\tau) \leq 1$, then from
Lemma \ref{l:h} it must be true that $h(\cdot,\tau)$ has a unique
global maximum which is strictly bigger than 1. By the Intermediate
Value Theorem there exists $k(\tau) \in [0,\infty)$ such that
$h(k(\tau),\tau) = 1$, and this justifies the existence of $k(\tau)$
satisfying \eqref{e:k(a)defn}. Since
there is only one stationary point of $h(\cdot,\tau)$ on
$[0,\infty)$ it follows that $k(\tau)$ is unique.

Suppose $k(\tau) \notin \mathbb{N}_0$. In this case, if $k \geq
k^*(\tau)$ then $h(k,\tau) > 1$ and $\beta_{k+1} > \beta_k$; that
is,
\begin{equation*}
\beta_{k^*(\tau)} < \beta_{k^*(\tau)+1} < \beta_{k^*(\tau)+2} < \cdots
\end{equation*}
Similarly, if $k \leq k^*(\tau)-1$ then $h(k,\tau) < 1$ and $\beta_{k+1}
< \beta_k$; that is,
\[
\beta_{k^*(\tau)} < \beta_{k^*(\tau)-1} < \cdots < \beta_0.
\]
This means, of course,
\begin{equation*}
\inf_{k \in \mathbb{N}_0} \beta_k = \beta_{k^*(\tau)}
\end{equation*}
and, using \eqref{e:asubstar},
\begin{equation*}
 \sup_{k \in \mathbb{N}_0} \beta_k = \max\{\beta_0,\lim_{k \to \infty} \beta_k\} = \left\{\begin{array}{llllll}
\lim_{k \to \infty} \beta_k & \text{if $\tau \in [\tau_*,\tau^*)$}  \\
\beta_0 & \text{if $\tau \in [\tau^*,d)$}
\end{array} \right..
\end{equation*}
By a similar argument, the same conclusion is true in the case
$k(\tau) \in \mathbb{N}_0$, except the infimum is not uniquely
attained because $h(k^*(\tau),\tau) = 1$ means $\beta_{k^*(\tau)+1}
= \beta_{k^*(\tau)}$.

The proof of Theorem \ref{t:betak} will be complete once we
verify that $k(\tau) \in (0,1)$ and $k^*(\tau) = 1$ whenever $d \geq 6$ and $\tau \in
(\tau_*,d)$. Since $h(0,\tau) < 1$ for such $\tau$, it suffices
to show that $h(1,\tau) > 1$. For this, we define
$$
\Theta(\tau) := h(1,\tau) = \frac{d+2-\tau}{d+\tau} \bigg(\frac{2d+1}{d}\bigg)^{\frac{\tau-1}{2}}.
$$
Necessarily $B_0(d,\tau) < 0$; otherwise each $B_j(d,\tau) > 0$ (see the proof of Lemma \ref{l:h})
which implies $h(\cdot,\tau)$ is decreasing on $(0,\infty)$. Since $h(0,\tau) < 1$ this means $\lim_{k \to \infty} h(k,\tau) < 1$, which is false.

Now $B_0(d,\tau) < 0$ if and only if $\tau \in
(\tau(d),d)$, where $\tau(d)$ is the largest root of
$$
\tau(2-\tau) = - \frac{(d-4)(d^2-d+2)}{d-3};
$$
that is,
$$
\tau(d) = 1 + \sqrt{1 + \frac{d-4}{d-3}(d^2-d+2)}\,.
$$
Define $\eta(d) = d-\tau(d)$ to be the length of the
interval for which $B_0(d,\tau) < 0$. Then it is straightforward to
check that
\begin{align*}
\eta(d)  = d-1-\sqrt{1 + \frac{d-4}{d-3}(d^2-d+2)} = \frac{8}{(d-3)\bigg(d-1+\sqrt{1 + \frac{d-4}{d-3}(d^2-d+2)}\bigg)} \leq \eta(6),
\end{align*}
where
$$
\eta(6) = \frac{8}{3(5+\sqrt{67/3})} = 0.274...
$$
We shall prove that
\begin{equation} \label{e:d=6main}
0 \leq \eta \leq \eta(d) \qquad \Longrightarrow \qquad \Theta(d-\eta) > 1.
\end{equation}
To establish \eqref{e:d=6main} first notice that
\begin{align*}
\Theta(d-\eta) = \frac{2+\eta}{2d-\eta} \bigg(\frac{2d+1}{d}\bigg)^{\frac{d-\eta-1}{2}}
\geq \frac{1}{d} \,2^{\frac{d-\eta-1}{2}} \bigg(1+\frac{1}{2d}\bigg)^{\frac{d-\eta-1}{2}}
\end{align*}
and therefore
\begin{equation*}
\Theta(d-\eta) \geq 2^{\frac{d-\eta-1}{2}} \bigg( \frac{5d-\eta-1}{4d^2} \bigg) \geq 2^{\frac{d-\eta(6)-1}{2}} \bigg( \frac{5d-\eta(6)-1}{4d^2} \bigg)
\end{equation*}
since $0 \leq \eta \leq \eta(d) \leq \eta(6)$. This means we will
have shown \eqref{e:d=6main} once we show that
\begin{equation} \label{e:d=6do}
2^{\frac{d-\eta(6)-1}{2}} > \frac{4d^2}{5d-\eta(6)-1}.
\end{equation}
When $d=6$, we have $2^{\frac{5-\eta(6)}{2}} = 5.144...$ and
$\frac{144}{29-\eta(6)} = 5.012...$ and hence \eqref{e:d=6do} is
true. We will show that \eqref{e:d=6do} is true for larger
dimensions using a straightforward induction argument, and so we
assume \eqref{e:d=6do} is true for some fixed $d \geq 6$. Using this
assumption,
\begin{equation*}
2^{\frac{d-\eta(6)}{2}} > \frac{4 \sqrt{2}d^2}{5d-\eta(6)-1}
\end{equation*}
and so it suffices to check that
\begin{equation} \label{e:d=6poly}
\frac{4 \sqrt{2}d^2}{5d-\eta(6)-1} \geq \frac{4(d+1)^2}{5(d+1)-\eta(6)-1}.
\end{equation}
It is clear that if $d \geq (2^{1/4} - 1)^{-1}$ then $\sqrt{2}d^2
\geq (d+1)^2$, and hence \eqref{e:d=6poly} holds for such $d$. But
$(2^{1/4} - 1)^{-1} = 5.285...$, therefore \eqref{e:d=6poly}, and
hence \eqref{e:d=6main}, is true for all $d \geq 6$.

Bringing everything together, whenever $d \geq 6$ and $\tau \in
(\tau_*,d)$, with $\eta = d-\tau$ then we have $0 < \eta < \eta(d)$
and consequently \eqref{e:d=6main} implies $h(1,\tau) = \Theta(\tau)
> 1$.

From the above proof, one can easily extract the claimed expressions
for the index sets $\mathbf{k}(d,\tau;\theta)$ and
$\mathbf{K}(d,\tau;\theta)$ on account of Theorem \ref{t:Main}; we omit the details.
\end{proof}

\begin{proof}[Proof of Proposition \ref{p:growth}]
To obtain the claimed lower bound, define
$$
\Delta_k(\varepsilon) = \frac{\partial}{\partial \varepsilon}(h(k,5- 2\varepsilon))
$$
for $\varepsilon \geq 0$. Then
\begin{align*}
\Delta_k(\varepsilon) & = \bigg[ \frac{2(k+2)}{(k+4-\varepsilon)^2} - \frac{k+\varepsilon}{k+4-\varepsilon} \log\bigg(\frac{k^2 + 5k + 5}{k^2+3k+1} \bigg)\bigg] \bigg(\frac{k^2 + 5k + 5}{k^2+3k+1} \bigg)^{2-\varepsilon} \\
& \leq \frac{2(k+2)}{(k+3)^2} \bigg(\frac{k^2 + 5k + 5}{k^2+3k+1} \bigg)^2
\end{align*}
for $0 < \varepsilon < 1$, and it follows from the Mean Value
Theorem that
\begin{equation*}
h(k,5-2\varepsilon) \leq h(k,5) + \varepsilon \frac{2(k+2)}{(k+3)^2} \bigg(\frac{k^2 + 5k + 5}{k^2+3k+1} \bigg)^2.
\end{equation*}
Hence
\begin{equation} \label{e:lowerimp}
0 \leq \varepsilon < \varepsilon(k) \qquad \Longrightarrow \qquad  h(k,5-2 \varepsilon) < 1,
\end{equation}
where
$$
\varepsilon(k) := (1-h(k,5)) \,\frac{(k+3)^2(k^2+3k+1)^2}{2(k+2)(k^2+5k+5)^2}\,.
$$
A straightforward calculation shows that
$$
1-h(k,5) = 1 - \frac{k(k^2+5k+5)^2}{(k+4)(k^2+3k+1)^2} = \frac{4}{(k+4)(k^2+3k+1)^2}
$$
and therefore
$$
\varepsilon(k) = \frac{2(k+3)^2}{(k+2)(k+4)(k^2+5k+5)^2} \geq \frac{C}{(k+1)^4}
$$
for some absolute constant $C$ and all $k \in \mathbb{N}_0$.

For fixed $\tau \in (\tau_*,5)$, if we take $k \in
\mathbb{N}_0$ such that
$$
k < \bigg(\frac{2C}{5-\tau}\bigg)^{1/4} - 1
$$
then $\frac{1}{2}(5-\tau) < \varepsilon(k)$ and, by \eqref{e:lowerimp},
we get $h(k,\tau) < 1$. This means that $k(\tau) \in \mathbb{N}_0$
satisfying \eqref{e:k(a)defn} satisfies the lower bound
\begin{equation*}
k(\tau) \geq \bigg(\frac{2C}{5-\tau}\bigg)^{1/4} - 1.
\end{equation*}

For the upper bound, by Lemma \ref{l:h} we make the observation that $k(\tau)$ cannot
exceed the positive value of $k$ at which $\frac{\partial
h}{\partial k}(k,\tau)$ is equal to zero, given by
\begin{equation*}
k = -2 + \sqrt{\frac{B_1(5,\tau) - B_0(5,\tau)}{B_2(5,\tau)}}\,;
\end{equation*}
that is
\begin{equation*}
k(\tau) \leq -2 + \sqrt{\frac{5-\tau(2-\tau)}{(5-\tau)(3+\tau)}}\,.
\end{equation*}
Hence, there exists positive constants $C_1$ and $C_2$ such that,
for all $\tau \in (\tau_*,5)$,
\begin{equation*}
\frac{C_1}{(5-\tau)^{1/4}} \leq k(\tau) \leq \frac{C_2}{(5-\tau)^{1/2}},
\end{equation*}
as claimed.
\end{proof}

\section{Further results}
\label{section:homogeneousg}

We begin by considering the case $\theta(\rho) = \rho^{\frac{\tau-1}{4}}$, and $\phi$ and $\psi$ satisfying \eqref{e:phipsi}, with $d \geq 2$ and $\tau \in (1,d)$. Then we have
\[
\beta_k = \pi 2^{2-\tau} \frac{\Gamma(\tau - 1)\Gamma(k+\tfrac{d-\tau}{2})}
{\Gamma(\tfrac{\tau}{2})^2\Gamma(k+\tfrac{d+\tau}{2}-1)} \,(k(k+d-2))^{\frac{\tau-1}{2}}
\]
and it is clear that $\mathbf{b}(d,\tau;\theta) = 0$. Also, if
\[
h(k,\tau) = \frac{\beta_{k+1}}{\beta_k} = \frac{2k+d-\tau}{2k+d+\tau-2} \,
\p{\frac{(k+1)(k+d-1)}{k(k+d-2)}}^{\frac{\tau-1}{2}}
\]
then
\[
\frac{\partial h}{\partial k}(k,\tau)=-A(d,k,\tau)\b{B_0(d,\tau)+B_1(d,\tau)k+B_2(d,\tau)k^2},
\]
where
\begin{align*}
A(d,k,\tau)&=
\frac{\tau - 1} {2(2k+d+\tau-2)^2 \{k(k+d-2)\}^{2}}
\p{\frac{(k+1)(k+d-1)}{k(k+d-2)}}^{\frac{\tau-3}{2}},
\\
B_0(d,\tau)&=(d-1)(d-2)(d-2+\tau)(d-\tau),
\\
B_1(d,\tau)&=2(d-1)(3(d-2)^2+(2-\tau)\tau),
\\
B_2(d,\tau)&=2(3 (d-2)^2+(2-\tau)\tau).
\end{align*}
For $d \geq 2$ and $\tau \in (1,d)$, obviously we have
$A(k,d,\tau)>0$ and $B_0(d,\tau) \geq 0$ .

Since $(2-\tau)\tau$ is strictly decreasing for $\tau \in (1,d)$, we
have $3(d-2)^2+(2-\tau)\tau
> 2(d-2)(d-3)$ and therefore $B_1(d,\tau)
> 0$ and $B_2(d,\tau)>0$. This means $h(\cdot,\tau)$ is
strictly decreasing and tends to $1$ from above. It follows that
$(\beta_k)_{k \in \mathbb{N}_0}$ is strictly increasing and
\[
\mathbf{B}(d,\tau;\theta)=\lim_{k\to\infty} \beta_k = \pi 2^{2-\tau} \frac{\Gamma(\tau - 1)}{\Gamma(\frac{\tau}{2})^2}.
\]

Using Theorem \ref{t:Main} we may use the above analysis to obtain
the following.
\begin{corollary}
Let $d \geq 2$, $s \in (-\frac{1}{2},\frac{d}{2}-1)$ and suppose
that $i\partial_t u + \frac{1}{2}\Delta u = 0$ on
$\mathbb{R}^{d+1}$. Then
\[
\int_{\mathbb{R}} \int_{\mathbb{R}^d}
|(-\Lambda)^{\frac{1+2s}{4}}u(x,t)|^2 \, \frac{\mathrm{d}x\mathrm{d}t}{|x|^{2(1+s)}}  \leq
\pi 2^{-2s} \frac{\Gamma(2s+1)}{\Gamma(s+1)^2}\, \|u(0)\|_{\dot{H}^s(\mathbb{R}^d)}^2,
\]
the constant is optimal and there are no extremisers.
\end{corollary}
\begin{corollary}
Let $d \geq 2$, $s \in (0,\frac{d-1}{2})$ and suppose that
$\partial_{tt} u - \Delta u = 0$ on $\mathbb{R}^{d+1}$. Then
\begin{equation*}
  2\int_\mathbb{R} \int_{\mathbb{R}^d} |(-\Lambda)^{\frac{s}{2}}u(x,t)|^2 \, \frac{\mathrm{d}x\mathrm{d}t}{|x|^{1+2s}} \leq
  \pi 2^{1-2s} \frac{\Gamma(2s)}{\Gamma(s+\frac{1}{2})^2} \left(\|u(0)\|_{\dot{H}^s(\mathbb{R}^d)}^2 + \|\partial_tu(0)\|_{\dot{H}^{s-1}(\mathbb{R}^d)}^2\right),
\end{equation*}
the constant is optimal and there are no extremisers.
\end{corollary}
\begin{corollary}
Let $d \geq 2$, $s \in (-\frac{1}{2},\frac{d}{2}-1)$ and suppose
that $\partial_{tt} u - \Delta u + u = 0$ on $\mathbb{R}^{d+1}$.
Then
\[
2\int_{\mathbb{R}} \int_{\mathbb{R}^d} |(1-\Delta)^{\frac{1}{4}}(-\Lambda)^{\frac{1+2s}{4}}u(x,t)|^2 \, \frac{\mathrm{d}x\mathrm{d}t}{|x|^{2(1+s)}}
\leq \pi 2^{-2s} \frac{\Gamma(2s+1)}{\Gamma(s+1)^2} \,\|(u(0),\partial_tu(0))\|^2,
\]
the constant is optimal and there are no extremisers. Here, the norm
on the initial data is given by
\[
\|(u(0),\partial_tu(0))\|^2 = \|u(0)\|_{\dot{H}^s(\mathbb{R}^d)}^2 +
\|u(0)\|_{\dot{H}^{s+1}(\mathbb{R}^d)}^2 +
\|\partial_tu(0)\|_{\dot{H}^{s}(\mathbb{R}^d)}^2.
\]
\end{corollary}

We finish by stating a more general result than Theorem \ref{t:Main}, which is not restricted to homogeneous weights, and does not require the dispersion relation and smoothing functions $\phi$ and $\psi$ to satisfy \eqref{e:phipsi}. The cost of this generality is that the optimal constants are less explicit and precise information about the extremisers is less readily available.
\begin{theorem} \label{t:mostgeneral}
Suppose $w : [0,\infty) \to [0,\infty)$, $\psi : [0,\infty) \to [0,\infty)$ and
$\phi : [0,\infty) \to \mathbb{R}$ are such that $\alpha_k : [0,\infty) \to [0,\infty)$ is
continuous for each $k \in \mathbb{N}_0$, where
\[
\alpha_k(\rho) = \frac{\rho \psi(\rho)^2}{|\phi'(\rho)|} \int_0^\infty J_{\nu(k)}(r \rho)^2 rw(r) \, \mathrm{d}r\,,
\]
$\nu(k) = \frac{d}{2} + k -1$ and $J_{\nu(k)}$ is the Bessel function of the first kind with order $\nu(k)$. If $i\partial_t u + \phi(|\nabla|)u = 0$ then
\begin{equation*}
\inf_{k\in\N_0} \inf_{\rho > 0} \beta_k(\rho) \, \|u(0)\|_{L^2(\mathbb{R}^d)}^2 \leq
\int_{\mathbb{R}} \int_{\mathbb{R}^d} |\psi(|\nabla|) \,\theta(-\Lambda) \,u(t,x)|^2 w(|x|) \, \mathrm{d}x\mathrm{d}t
\leq \sup_{k\in\N_0} \sup_{\rho > 0} \beta_k(\rho) \, \|u(0)\|_{L^2(\mathbb{R}^d)}^2\,,
\end{equation*}
where 
\[
\beta_k(\rho) = 2\pi |\theta(k(k+d-2))|^2 \alpha_k(\rho)
\]
and the constants are optimal.
\end{theorem}
In the case where the weight $w$ is homogeneous, but $\phi$ and $\psi$ do not necessarily satisfy \eqref{e:phipsi}, we may deduce the following.
\begin{corollary} \label{c:generalhomogeneous}
Let $\tau \in (1,d)$. Suppose $\psi : [0,\infty) \to [0,\infty)$ and
$\phi : [0,\infty) \to \mathbb{R}$ are such that $\zeta : [0,\infty) \to [0,\infty)$ is
continuous, where
\[
\zeta(\rho) = \frac{\rho^{\tau-1}\psi(\rho)^2}{|\phi'(\rho)|}\,.
\]
Let 
\[
\beta_k = \pi 2^{2-\tau} \frac{\Gamma(\tau - 1)\Gamma(k+\tfrac{d-\tau}{2})}
{\Gamma(\tfrac{\tau}{2})^2\Gamma(k+\tfrac{d+\tau}{2}-1)} \,|\theta(k(k+d-2))|^2
\]
for $k \in \mathbb{N}_0$. If $i\partial_t u + \phi(|\nabla|)u = 0$ then
\begin{equation*}
\inf_{k\in\N_0} \beta_k \inf_{\rho > 0} \zeta(\rho) \, \|u(0)\|_{L^2(\mathbb{R}^d)}^2 \leq
\int_{\mathbb{R}} \int_{\mathbb{R}^d} |\psi(|\nabla|) \,\theta(-\Lambda) \,u(t,x)|^2  \, \frac{\mathrm{d}x\mathrm{d}t}{|x|^\tau}
\leq \sup_{k\in\N_0} \beta_k \sup_{\rho > 0} \zeta(\rho) \, \|u(0)\|_{L^2(\mathbb{R}^d)}^2
\end{equation*}
and the constants are optimal.
\end{corollary}
It is clear that Corollary \ref{c:generalhomogeneous} extends the sharp estimates in Theorem \ref{t:Main}
since $\zeta$ is identically equal to one under the assumption \eqref{e:phipsi}. However, the situation
regarding extremisers is more complicated when \eqref{e:phipsi} does not hold. We may use Theorem 1.2 from
\cite{BS} to see that the existence of extremisers for the upper bound is equivalent to the existence of some $k_0 \in \mathbb{N}_0$ such that 
\[
\beta_{k_0} = \sup_{k \in \mathbb{N}_0} \beta_k
\]
and a subset $\mathcal{S}$ of $(0,\infty)$ with positive Lebesgue measure such that 
\[
\zeta(\rho_0) = \sup_{\rho > 0} \zeta(\rho) \qquad \text{for each $\rho_0 \in \mathcal{S}$.}
\]
An analogous remark is also valid for the lower bound, where each instance of $\sup$ is replaced by $\inf$.

Corollary \ref{c:generalhomogeneous} follows immediately from Theorem \ref{t:mostgeneral} and the formula
\[
\int_0^\infty J_{\nu(k)}(r\rho)^2 \, \frac{\mathrm{d}r}{r^{\tau-1}} = 2^{1-\tau} \frac{\Gamma(\tau-1)\Gamma(k + \frac{d-\tau}{2})}{\Gamma(\frac{\tau}{2})^2 \Gamma(k + \frac{d+\tau}{2} - 1)} \rho^{\tau - 2}.
\]
This formula was also used in the proof of Theorem 1.6 in \cite{BS}, and Corollary \ref{c:generalhomogeneous} is in fact a generalisation of Theorem 1.6 in \cite{BS} where the case $\theta$ identically equal to one is given.

In a similar manner, Theorem \ref{t:mostgeneral} is a generalisation of Theorem 4.1(b) of \cite{WaltherBest} where the case $\theta$ identically equal to one is given. To prove Theorem \ref{t:mostgeneral} we use our observation that the operator $S_\theta$, introduced in \eqref{e:Sdefn}, satisfies \eqref{e:Sg} and proceed using the same argument in \cite{WaltherBest}; we omit the details. We remark that this proof via the argument in \cite{WaltherBest} is also based on a spherical harmonic decomposition and orthogonality arguments, but does not yield precise spectral information as in Proposition \ref{t:Sg}.

\end{document}